\pdfoutput=1
\documentclass[11pt]{article}

\usepackage[a4paper,margin=28mm]{geometry}

\usepackage[T1]{fontenc}
\usepackage[utf8]{inputenc}
\usepackage{lmodern}
\usepackage{amsmath,amssymb,amsthm,mathtools}
\usepackage{bm}
\usepackage{mathrsfs}

\usepackage{graphicx}
\usepackage{booktabs}
\usepackage{array}
\usepackage{enumitem}
\usepackage{comment}

\usepackage{aliascnt}% robust theorem-type detection for cleveref across TeX distributions
\usepackage[colorlinks=true,linkcolor=blue,citecolor=blue,urlcolor=blue]{hyperref}
\usepackage[nameinlink,capitalize]{cleveref}

\theoremstyle{plain}
\newtheorem{theorem}{Theorem}[section]

\newaliascnt{lemma}{theorem}
\newtheorem{lemma}[lemma]{Lemma}
\aliascntresetthe{lemma}

\newaliascnt{proposition}{theorem}
\newtheorem{proposition}[proposition]{Proposition}
\aliascntresetthe{proposition}

\newaliascnt{corollary}{theorem}
\newtheorem{corollary}[corollary]{Corollary}
\aliascntresetthe{corollary}

\theoremstyle{definition}

\newaliascnt{definition}{theorem}
\newtheorem{definition}[definition]{Definition}
\aliascntresetthe{definition}

\newaliascnt{assumption}{theorem}

\aliascntresetthe{assumption}

\newaliascnt{example}{theorem}
\newtheorem{example}[example]{Example}
\aliascntresetthe{example}

\theoremstyle{remark}

\newaliascnt{remark}{theorem}
\newtheorem{remark}[remark]{Remark}
\aliascntresetthe{remark}

\crefname{assumption}{Assumption}{Assumptions}
\Crefname{assumption}{Assumption}{Assumptions}

\newcommand{\R}{\mathbb{R}}
\newcommand{\C}{\mathbb{C}}
\newcommand{\N}{\mathbb{N}}

\newcommand{\Tend}{\mathfrak{T}}          % terminal time

\newcommand{\norm}[1]{\left\lVert #1 \right\rVert}

\newcommand{\dual}[2]{\left\langle #1,#2 \right\rangle}
\newcommand{\diff}{\mathop{}\!\mathrm{d}}
\newcommand{\Km}{K_{\mu}}
\newcommand{\Dm}{\mathcal{D}_{\mu}}
\newcommand{\Lp}{L^{2}(\Omega)^{d}}
\newcommand{\LpC}{L^{2}(\Omega;\C)^{d}}

\renewcommand{\Re}{\operatorname{Re}}
\renewcommand{\Im}{\operatorname{Im}}

\makeatletter

\@addtoreset{equation}{section}
\makeatother
\title{High-frequency coercivity loss for completely monotone memory on bounded time intervals}

\author{
  Hiroki Ishizaka\\
  Team FEM, Matsuyama, Japan\\
  E-mail: \texttt{h.ishizaka005@gmail.com}
}

\date{}

\begin{document}

\maketitle

\begin{abstract}
We study Volterra memory terms with locally integrable completely monotone kernels on a finite time interval $(0,\Tend)$ and ask how much $L^{2}$ coercivity they retain at high temporal frequencies. The main tool is an exact formula for the Rayleigh quotients of the cosine modes $\psi_n(t)=(2/\Tend)^{1/2}\cos(2\pi nt/\Tend)$. It shows that these quotients lie between $(1-\kappa_n/(2\pi n))\,m(2\pi n/\Tend)$ and $m(2\pi n/\Tend)$, where $m$ is the real memory symbol and $\kappa_n\in[1-e^{-2\pi n},1]$ is the best constant valid for all such kernels. Consequently, for non-constant kernels the algebraic decay rate of the quotients does not depend on $\Tend$ and coincides with the decay index $\rho\in[0,2]$ of $m$; any value in $[0,2]$ occurs. The Gaussian kernel $e^{-t^{2}}$, which is of positive type but not completely monotone, shows that this may fail otherwise: its quotients decay like $n^{-4}$, while its symbol decays faster than any power. We also show that the largest Rayleigh quotient on $(0,\Tend)$ is a continuous and strictly increasing function of $\Tend$, so that a kernel of total mass larger than one has exactly one critical horizon. Finally, because the memory operator on $(0,\Tend)$ is compact, it does not contribute to the uniform $L^{2}$ coercivity constant, and for $k_n(t)=ne^{-nt}$ the associated operators converge to the identity strongly but not in norm. The diffusion equation with memory serves as a model throughout.
\end{abstract}



%\tableofcontents

% ============================================================
\section{Introduction}
% ============================================================
In classical diffusion, the flux depends on the current gradient. In a material with memory, it also depends on past gradients. Our model problem is the following. Let $\Omega\subset\R^{d}$, $d\in\N$, be a bounded Lipschitz domain, let $H:=L^{2}(\Omega)$ and $V:=H^{1}_{0}(\Omega)$, and consider
\begin{align}\label{eq:model}
\partial_{t}u(t)+\mathsf{A}_{0}u(t)+\mathsf{A}_{1}(k*u)(t)=f(t)
\quad\text{in }V',
\quad
u(0)=u_{0},
\end{align}
on a finite interval $(0,\Tend]$. Here, $\mathsf{A}_{0}$ and $\mathsf{A}_{1}$ are induced by bounded symmetric forms $a_{0}$ and $a_{1}$, and $k$ is a locally integrable completely monotone kernel. Problems of this type arise in heat conduction with memory \cite{GurtinPipkin1968}, viscoelasticity and fading-memory materials \cite{Dafermos1970,AmendolaFabrizioGolden2012,RenardyHrusaNohel1987}, and viscoelastic-fluid models in which the solvent-viscosity contribution may disappear \cite{WangRenardy2011}.

More generally, let $X$ be a real Hilbert space and let $B\in\mathcal L(V,X)$, so that
\begin{align*}
a_{1}(w,v)=(Bw,Bv)_{X},\quad W(t):=Bu(t).
\end{align*}
For the diffusion realisation, $X = L^{2}(\Omega)^{d}$ and $Bv =A_{1}^{1/2}\nabla v$. The memory dissipation then depends on $u$ only through $W\in L^{2}(0,\Tend;X)$. All questions below therefore concern a convolution operator in time acting on $X$-valued functions, and the spatial structure plays no further role.

Positive-definite Volterra kernels and their frequency-domain characterisation are classical. MacCamy and Wong \cite{MacCamyWong1972}, MacCamy \cite{MacCamy1974}, Staffans \cite{Staffans1975,Staffans1976Measures,Staffans1976Inequality}, Nohel and Shea \cite{NohelShea1976}, and Londen \cite{Londen1977,Londen1980} developed positivity and frequency methods for scalar and Hilbert-space Volterra equations. Staffans' measure-theoretic framework already includes point masses at the origin \cite{Staffans1976Measures}, and his later viscoelastic theory separates instantaneous viscosity from positive-type distributed relaxation \cite{Staffans1994}; see also the monographs \cite{GripenbergLondenStaffans1990,Pruss1993}, and \cite{Zacher2009} for weak solutions of abstract evolutionary integro-differential equations in Hilbert spaces. None of these ingredients is new here, and we use them freely: positive type, the Fourier representation, the internal-variable construction, and the elementary facts that $L^{1}$ kernels give compact operators and that a compact operator has distance at least $a$ from $aI$.

Integral operators with difference kernels on a bounded interval also have a well-developed spectral theory; Rosenblatt \cite{Rosenblatt1963} and Widom \cite{Widom1963,Widom1964} obtained asymptotic formulas for their ordered eigenvalues. Our estimates are of a different kind. They concern one explicit family of test functions rather than the eigenvalues, but they hold for every mode and every $\Tend>0$, with constants independent of the kernel and without any regular-variation assumption. \Cref{rem:kyfan} records what they give for partial sums of eigenvalues.

Our first result concerns the cosine modes. For
\begin{align*}
\omega_n:=\frac{2\pi n}{\Tend},\quad
\psi_n(t):=\sqrt{\frac{2}{\Tend}}\cos(\omega_nt),
\end{align*}
we compute $\mathcal Q_{k,\Tend}[\psi_n]$ explicitly and compare it with the symbol $m(\omega_n)$. With $a_n:=2\pi n$ and
\begin{align*}
\kappa_n:=\max_{x\geq0}\frac{2a_nx(1-e^{-x})}{x^{2}+a_n^{2}},
\quad
1-e^{-a_n}\leq\kappa_n\leq1,
\end{align*}
we show that
\begin{align*}
\left(1-\frac{\kappa_n}{a_n}\right)m(\omega_n)
\leq \mathcal Q_{k,\Tend}[\psi_n]
\leq m(\omega_n).
\end{align*}
Both constants are best possible, and the lower one is attained by a single exponential. In particular, $\mathcal Q_{k,\Tend}[\psi_n]/m(\omega_n)\to1$ for any non-constant kernel, with relative error at most $1/(2\pi n)$.

Because $m$ is monotone, this comparison transfers algebraic decay rates in both directions. Writing $\rho(k)$ for the decay index of $m$ and $\rho_{\Tend}^{\mathrm{cos}}(k)$ for that of $n\mapsto\mathcal Q_{k,\Tend}[\psi_n]$, we obtain for any non-constant completely monotone kernel
\begin{align*}
\rho_{\Tend}^{\mathrm{cos}}(k)=\rho(k)\in[0,2]
\quad\text{for any }\Tend>0,
\end{align*}
and any value in $[0,2]$ occurs, already for kernels with absolutely continuous representing measure. For kernels that are merely of positive type this can fail. The Gaussian $k_G(t)=e^{-t^2}$ satisfies
\begin{align*}
m_G(\omega)=\frac{\sqrt\pi}{2}e^{-\omega^2/4},
\quad
\mathcal Q_{k_G,\Tend}[\psi_n]\sim
\frac{4}{\Tend}\bigl[1+(2\Tend^2-1)e^{-\Tend^2}\bigr]\omega_n^{-4}.
\end{align*}
Therefore, $\rho(k_G)=+\infty$ but $\rho_{\Tend}^{\mathrm{cos}}(k_G)=4$ on $(0,\Tend)$ the boundary terms coming from the truncation dominate. For completely monotone kernels, these terms have a sign and are controlled mode by mode, and this is why the two indices agree. Because $x\mapsto m(\sqrt{x})$ is a Stieltjes transform, regularly varying symbols are covered by classical Karamata theory \cite{BinghamGoldieTeugels1987}; what is new is that $\rho$ can be read off from Rayleigh quotients on an arbitrary bounded interval.

Two simpler facts complete the picture. First, a positive-type kernel in $L^{1}(0,\Tend)$ defines a compact non-negative operator on $L^{2}(0,\Tend)$, so the memory term alone has coercivity constant zero, and adding $aI$ raises it to exactly $a$; the cosine estimates above say how fast the modes $\psi_n$ approach this bound. Second, at the top of the spectrum, the norm $\Lambda_k(\Tend)$ of the memory operator on $(0,\Tend)$ is continuous and strictly increasing in $\Tend$. If the total mass exceeds one, there is therefore a unique $\Tend_{\ast}$ with $\Lambda_k(\Tend_{\ast})=1$, and the memory term is dominated by the instantaneous term exactly for $\Tend\leq \Tend_{\ast}$.

Compactness also explains why the limit of concentrating kernels is singular: memory operators can converge strongly to $aI$, but never in norm. For $k_n(t)=ne^{-nt}$ and any $\Tend>0$,
\begin{align*}
\mathcal S_{k_n,\Tend}\to I\quad\text{strongly},
\quad
\|\mathcal S_{k_n,\Tend}-I\|=1,
\end{align*}
while $m_n(\omega)=n^{2}/(n^{2}+\omega^{2})$ tends to $1$ at each fixed frequency but still decays at infinity for any $n$. This complements singular-limit results for solutions and attractors \cite{ContiPataSquassina2006,Shikhman2026}, which concern convergence of solutions rather than of coercivity constants. Along the way we also record the frequency identities for kernels of finite mass and for fractional kernels.

\Cref{sec:prelim} sets up the abstract framework and the diffusion model. \Cref{sec:symbol} introduces the symbol and proves the cosine-mode estimates, the compactness statement and the properties of $\Lambda_k$. \Cref{sec:structure} contains the index theorem, the Gaussian example and the instantaneous limit.

\section{Notation and preliminaries}\label{sec:prelim}
For the temporal results, let $X$ be a real Hilbert space and let $B\in\mathcal L(V,X)$. We write
\begin{align}\label{eq:abstract-factorisation}
a_{1}(w,v):=(Bw,Bv)_{X},
\quad
W(t):=Bu(t).
\end{align}
Thus, $a_{1}$ is bounded, symmetric and non-negative, and every coercivity statement for the memory term depends only on the Hilbert-valued field $W$. The diffusion problem \cref{eq:model} is the following concrete realisation. Let $A_{0},A_{1}\in L^{\infty}(\Omega)^{d\times d}$ be symmetric matrix fields and define
\begin{align*}
a_{i}(w,v):=\int_{\Omega}A_{i}(x)\nabla w\cdot\nabla v\diff x,
\quad i\in\{0,1\}.
\end{align*}
We denote by $\mathsf A_i:V\to V'$ the induced spatial operator. We assume $A_{1}(x)\succeq0$ almost everywhere. Its symmetric square root $A_{1}^{1/2}$ is bounded and measurable, and the diffusion factorisation is
\begin{align}\label{eq:Wfield}
X&=\Lp,
\quad Bv =A_{1}^{1/2}\nabla v,\\
W(t)&=A_{1}^{1/2}\nabla u(t),
\quad a_{1}(u(s),u(t))=(W(s),W(t))_{X}.
\end{align}
In particular, $a_{1}(u(t),u(t))=\|W(t)\|_{X}^{2}$. We do not assume coercivity of $a_{0}$ unless explicitly stated. When $d=1$, the coefficient fields are scalar and $\nabla u=u'$. In the remainder, we formulate temporal statements whenever convenient for an arbitrary Hilbert-valued $W$; we retain the notation $\Lp$ when we return to the diffusion realisation.

\paragraph{Admissible states.}
For each $\Tend>0$, we define
\begin{align}
\mathcal A_{\Tend}
:=
\left\{
 u:(0,\Tend)\to V:
 u\text{ is strongly measurable and }Bu\in L^2(0,\Tend;X)
\right\}.
\label{eq:admissible-class}
\end{align}
Its elements are called \emph{energy-admissible states}. We also set
\begin{align}
\mathcal A_{\Tend}^{+}
:=
\left\{u\in\mathcal A_{\Tend}:\int_0^{\Tend}a_1(u(t),u(t))\diff t>0\right\}.
\label{eq:positive-energy-class}
\end{align}
In the diffusion factorisation \cref{eq:Wfield}, membership in $\mathcal A_{\Tend}$ is exactly the condition $A_1^{1/2}\nabla u\in L^2(0,\Tend;\Lp)$. Because $B\in\mathcal L(V,X)$,
\begin{align*}
L^2(0,\Tend;V)\subset\mathcal A_{\Tend},
\end{align*}
and no reverse inclusion is assumed when $a_1$ is not coercive.

\subsection{Completely monotone kernels}
We restrict attention to \emph{locally integrable} completely monotone kernels: completely monotone densities $k$ with $k\in L^{1}(0,\Tend)$ for any $\Tend>0$, equivalently with representing measure $\nu$ satisfying $\int_{[1,\infty)}\lambda^{-1}\diff\nu(\lambda)<\infty$. The associated measure is $\diff\mu(s)=k(s)\diff s$ on $(0,\Tend]$. This class contains the fractional kernels $t^{-\alpha}/\Gamma(1-\alpha)$, $\alpha\in(0,1)$, while excluding non-integrable completely monotone densities such as $k(t)=1/t$. Here, $\Gamma$ denotes the Gamma function, $\Gamma(z) := \int_0^{\infty} t^{z-1} e^{-t} \diff t$ for $z > 0$. We assume throughout that the kernel is non-trivial, $k\not\equiv0$; equivalently, its representing measure satisfies $\nu\neq0$.

\begin{definition}[Completely monotone kernel]\label{def:cm}
A function $k\colon(0,\infty)\to[0,\infty)$ is \emph{completely monotone} if it is of class $\mathcal{C}^{\infty}$ and $(-1)^{n}k^{(n)}(t) \geq 0$ for all $t>0$ and $n\in\N_{0}$. By Bernstein's theorem \cite[Thm.~1.4]{SchillingSongVondracek2012}, this holds if and only if there exists a non-negative Borel measure $\nu$ on $[0,\infty)$ such that
\begin{align}\label{eq:bernstein}
\displaystyle
k(t)=\int_{[0,\infty)} e^{-\lambda t}\diff\nu(\lambda), \quad t>0.
\end{align}
We call $\nu$ the representing measure of $k$ in Bernstein's theorem. The total mass of the kernel is $\norm{k}_{L^{1}(0,\infty)}=\int_{0}^{\infty}k(t)\diff t =\int_{[0,\infty)}\lambda^{-1}\diff\nu(\lambda)\in(0,\infty]$, by Tonelli's theorem.
\end{definition}

\begin{example}[Standard cases]\label{ex:kernels}
\leavevmode
\begin{itemize}[leftmargin=1.6em]
  \item \emph{Exponential (single relaxation time):} $k(t)=\gamma e^{-\gamma t}$ with $\gamma>0$. Then, $\nu=\gamma\,\delta_{\gamma}$ and $\norm{k}_{L^{1}(0,\infty)}=1$.
  \item \emph{Fractional:} $k(t)=t^{-\alpha}/\Gamma(1-\alpha)$ with $\alpha\in(0,1)$. Then, $\diff\nu_{\alpha}(\lambda) =\tfrac{\sin(\pi\alpha)}{\pi}\lambda^{\alpha-1}\diff\lambda$ and
        $\norm{k}_{L^{1}(0,\infty)}=\infty$; see, e.g., \cite{GorenfloKilbasMainardiRogosin2014} for the fractional-calculus background.
\end{itemize}
\end{example}

\begin{example}[Prony kernel / finite relaxation spectrum]\label{ex:prony}
Let $k(t)=\sum_{j=1}^{J}c_{j}e^{-\gamma_{j}t}$ with $c_{j}>0$ and distinct $\gamma_{j}>0$. Then, $k$ is completely monotone with representing measure $\nu=\sum_{j=1}^{J}c_{j}\delta_{\gamma_{j}}$, and
\begin{align*}
\displaystyle
&M_{0}=\sum_{j}c_{j}=k(0^{+}) := \lim_{t \downarrow 0} k(t),\quad
M_{1}=\sum_{j}c_{j}\gamma_{j}=-k'(0^{+}) := -  \lim_{t \downarrow 0} k'(t), \\
&\norm{k}_{L^{1}(0,\infty)}=\sum_{j}\frac{c_{j}}{\gamma_{j}}=m(0) \ \text{with }
  m(\omega)=\sum_{j}\frac{c_{j}\gamma_{j}}{\gamma_{j}^{2}+\omega^{2}}.
\end{align*}
Provided that $a_1\not\equiv0$, \cref{thm:gap} shows that the exact gap is non-negative for every $\Tend>0$ and every $u\in\mathcal A_{\Tend}$ if and only if
\begin{align*}
\displaystyle
\sum_j \frac{c_j}{\gamma_j}\leq1,
\end{align*}
and $m(\omega)=O(\omega^{-2})\to0$ (\cref{lem:symbol}), which is the high-frequency decay underlying \cref{thm:nogo}. The single-exponential kernel of \cref{ex:kernels} is the case $J=1$, $c_{1}=\gamma_{1}=\gamma$, for which $M_{0}=\gamma$, $M_{1}=\gamma^{2}$, and $\norm{k}_{L^{1}(0,\infty)}=1$ (the borderline of \cref{thm:gap}). This example illustrates that the zeroth Bernstein moment $M_{0}=\sum_{j}c_{j}$ and the $L^{1}$ mass $\norm{k}_{L^{1}(0,\infty)}=\sum_{j}c_{j}/\gamma_{j}$ are different quantities.
\end{example}

Because $k$ is locally integrable, $\mu$ is finite on $(0,\Tend]$: $\mu((0,\Tend])=\int_{0}^{\Tend}k(s)\diff s<\infty$. For $u\in\mathcal A_{\Tend}$, we use zero prehistory, extending $u$ by zero to negative times, and define the memory operator by
\begin{align}\label{eq:Kmu}
\displaystyle
\dual{(\Km u)(t)}{v}:=\int_{0}^{t}k(t-s)\,a_{1}\left(u(s),v\right)\diff s,
  \quad v\in V,
\end{align}
and the cumulative memory dissipation is
\begin{subequations} \label{eq:Dmu}
\begin{align}
\displaystyle
\Dm[u](\Tend) &:=\int_{0}^{\Tend}\dual{(\Km u)(t)}{u(t)}\diff t
  =\int_{0}^{\Tend}\left((k*W)(t),\,W(t)\right)\diff t, \label{eq:Dmu=a} \\
 (k*W)(t) &:=\int_{0}^{t}k(t-s)W(s)\diff s, \label{eq:Dmu=b}
\end{align}
\end{subequations}
where the second equality uses \cref{eq:Wfield}.

\begin{definition}[Positive type]\label{def:positive-type}
Let $X$ be a real Hilbert space and let $k\in L^{1}(0,\Tend)$. We say that $k$ is \emph{of positive type on $(0,\Tend)$} if
\begin{align*}
\int_{0}^{\Tend}
\left(
\int_{0}^{t}k(t-s)W(s)\diff s,
W(t)
\right)_{X}\diff t
\geq0
\end{align*}
for any $W\in L^{2}(0,\Tend;X)$. It is of positive type on finite horizons if this holds for any $\Tend>0$. In \cref{lem:posrep}, we verify this property for any locally integrable completely monotone kernel.
\end{definition}

\subsection{Internal variables}
For each $\lambda \geq 0$, we define the internal variable
\begin{align*}
\displaystyle
Z(\lambda,t):=\int_{0}^{t} e^{-\lambda(t-s)}W(s)\diff s,
\end{align*}
which leads to
\begin{align}\label{eq:internal}
\partial_{t}Z(\lambda,t)+\lambda Z(\lambda,t)=W(t), \quad Z(\lambda,0)=0.
\end{align}
By \cref{eq:bernstein}, Fubini's theorem and Young's inequality, the superposition is well defined as an $L^{2}(0,\Tend;\Lp)$-valued Bochner integral and
\begin{align}\label{eq:conv-internal}
 k*W=\int_{[0,\infty)} Z(\lambda,\cdot)\diff\nu(\lambda)
 \quad\text{in }L^{2}(0,\Tend;\Lp).
\end{align}
Indeed, with $(1-e^{-\lambda\Tend})/\lambda:=\Tend$ at $\lambda=0$,
\begin{align*}
\int_{[0,\infty)}\norm{Z(\lambda,\cdot)}_{L^{2}(0,\Tend;\Lp)}\diff\nu(\lambda)
&\leq \norm{W}_{L^{2}(0,\Tend;\Lp)}
\int_{[0,\infty)}\frac{1-e^{-\lambda\Tend}}{\lambda}\diff\nu(\lambda)\\
&=\norm{W}_{L^{2}(0,\Tend;\Lp)}\norm{k}_{L^{1}(0,\Tend)}<\infty.
\end{align*}
Consequently, \cref{eq:conv-internal} also holds for almost every $t\in(0,\Tend)$ after choosing representatives. This diffusive (extended-variable) representation is classical in the theory of materials with memory \cite{Dafermos1970,AmendolaFabrizioGolden2012}. We use it as the bridge between the time domain and the frequency domain.

\paragraph{Positive-type dissipation.}
The Bernstein representation expresses the memory as a superposition of relaxation modes. The resulting internal-variable energy is non-negative, but it is not the instantaneous energy $\int_0^{\Tend}a_1(u,u)\diff t$.

\begin{lemma}[Internal-variable representation]\label{lem:posrep}
For every $u\in\mathcal A_{\Tend}$,
\begin{align}\label{eq:posrep}
\displaystyle
\Dm[u](\Tend)
= \int_{[0,\infty)}\left[\frac{1}{2}\norm{Z(\lambda,\Tend)}_{\Lp}^{2}
     +\lambda\int_{0}^{\Tend}\norm{Z(\lambda,t)}_{\Lp}^{2}\diff t \right]\diff\nu(\lambda)
  \ \geq \ 0.
\end{align}
In particular, the kernel is of positive type. The identity depends on $u$ only through the field $W$ by \cref{eq:Dmu=a}, and with $Z$ given by the same formula, it holds verbatim for any $W\in L^{2}(0,\Tend;\Lp)$.
\end{lemma}

\begin{proof}
Because $k\in L^{1}(0,\Tend)$ and $W\in L^{2}(0,\Tend;\Lp)$, Young's inequality gives $k*W\in L^{2}(0,\Tend;\Lp)$, so $\Dm[u](\Tend)$ is well defined and finite. We prove the identity for the truncated measure $\nu_{\Lambda}:=\nu|_{[0,\Lambda]}$, with kernel $k_{\Lambda}(t):=\int_{[0,\Lambda]}e^{-\lambda t}\diff\nu(\lambda)$, and then let $\Lambda\to\infty$.

The measure $\nu_{\Lambda}$ is finite: for any fixed $\tau>0$, $e^{-\Lambda\tau}\nu([0,\Lambda])\leq\int_{[0,\Lambda]}e^{-\lambda\tau}\diff\nu(\lambda)\leq k(\tau)<\infty$. Furthermore $k_{\Lambda}(t)\uparrow k(t)$ pointwise by monotone convergence in $\lambda$, so $0\leq k(t)-k_{\Lambda}(t)\leq k(t)$ with $k\in L^{1}(0,\Tend)$, and dominated convergence gives $\norm{k-k_{\Lambda}}_{L^{1}(0,\Tend)}\to0$.

Fix $\lambda\geq0$. Because $e^{-\lambda\cdot}\in L^{1}(0,\Tend)$, Young's inequality gives $Z(\lambda,\cdot)\in L^{2}(0,\Tend;\Lp)$. Furthermore, $\partial_{t}Z(\lambda,\cdot)=W-\lambda Z(\lambda,\cdot)\in L^{2}(0,\Tend;\Lp)$, so $Z(\lambda,\cdot)\in H^{1}(0,\Tend;\Lp)$ solves \cref{eq:internal} in the $H^{1}$-trace sense. Since $\nu_{\Lambda}$ is finite and
\begin{align*}
\norm{Z(\lambda,t)}_{\Lp}
\leq\Tend^{1/2}\norm{W}_{L^{2}(0,\Tend;\Lp)},
\end{align*}
Fubini's theorem and $W=\partial_{t}Z+\lambda Z$ (using $Z(\lambda,0)=0$) give
\begin{align*}
D_{\mu_{\Lambda}}[u](\Tend)
&=\int_{[0,\Lambda]}\int_{0}^{\Tend}\left(Z(\lambda,t),\partial_{t}Z(\lambda,t)+\lambda Z(\lambda,t)\right)_{\Lp}\diff t\,\diff\nu(\lambda)\\
&=\int_{[0,\Lambda]}\left[\frac{1}{2}\norm{Z(\lambda,\Tend)}_{\Lp}^{2}+\lambda\int_{0}^{\Tend}\norm{Z(\lambda,t)}_{\Lp}^{2}\diff t\right]\diff\nu(\lambda).
\end{align*}
Finally, $k_{\Lambda}\to k$ in $L^{1}(0,\Tend)$ and Young's inequality give
\begin{align*}
|\Dm[u](\Tend)-D_{\mu_{\Lambda}}[u](\Tend)|
\leq\norm{k-k_{\Lambda}}_{L^{1}(0,\Tend)}\norm{W}_{L^{2}(0,\Tend;\Lp)}^{2}\to0,
\end{align*}
while the integrand above is non-negative, so the monotone convergence theorem passes $\Lambda\to\infty$ on the right. Combining the two limits yields \cref{eq:posrep}. The right-hand side is non-negative, and therefore $D_{\mu}[u](\Tend) \geq 0$.
\end{proof}

\begin{example}[Single exponential: time-domain gap]\label{ex:exp}
Let $k(t)=\gamma e^{-\gamma t}$ with $\gamma>0$ and $Z(t):=Z(\gamma,t)$. Then,
\begin{align}\label{eq:expgap}
\int_{0}^{\Tend} a_{1}(u(t),u(t))\diff t-\Dm[u](\Tend)
= \int_{0}^{\Tend}\norm{\partial_{t}Z(t)}_{\Lp}^{2}\diff t
 +\frac{\gamma}{2}\norm{Z(\Tend)}_{\Lp}^{2}\geq0.
\end{align}
Indeed, $\partial_tZ+\gamma Z=W$, $Z(0)=0$, and $(k*W)=\gamma Z$. Therefore,\cref{lem:posrep} gives
\begin{align*}
\Dm=\frac\gamma2\|Z(\Tend)\|^2+\gamma^2\int_0^{\Tend}\|Z\|^2\diff t,
\end{align*}
while expansion of $\|W\|^2=\|\partial_tZ+\gamma Z\|^2$ gives
\begin{align*}
\int_0^{\Tend}\|W\|^2\diff t
=\int_0^{\Tend}\|\partial_tZ\|^2\diff t
+\gamma\|Z(\Tend)\|^2
+\gamma^2\int_0^{\Tend}\|Z\|^2\diff t.
\end{align*}
\end{example}

% ============================================================
\section{The memory symbol and coercivity on bounded intervals}\label{sec:symbol}
% ============================================================
\Cref{lem:posrep} shows that the memory dissipation is non-negative, but says nothing about its size compared with $\|W\|^{2}$. We first describe the memory in the frequency domain and then return to a bounded interval, where compactness settles the question of a uniform lower bound under positive type alone.

\subsection{The symbol and its qualitative properties}
The positive-type identity of \cref{lem:posrep} is a real quadratic form in the history field $W$. After zero extension to $\R$, its frequency representation is governed by the causal Fourier transform of $k$. Only the real part contributes to the quadratic form.

We regard $k$ as a ``causal'' kernel by extending it by zero to negative times. Its Fourier transform is therefore
\begin{align*}
\displaystyle
\hat{k}(\omega)
:=
\int_0^\infty k(t)e^{-i\omega t}\,\diff t .
\end{align*}
Using \cref{eq:bernstein} and formally interchanging the order of integration, we obtain
\begin{align*}
\displaystyle
\hat{k}(\omega)
=
\int_{[0,\infty)}
\int_0^\infty e^{-(\lambda+i\omega)t}\,\diff t\,\diff\nu(\lambda)
=
\int_{[0,\infty)}
\frac{1}{\lambda+i\omega}\,\diff\nu(\lambda).
\end{align*}
Because
\begin{align*}
\displaystyle
\operatorname{Re}
\frac{1}{\lambda+i\omega}
=
\frac{\lambda}{\lambda^2+\omega^2},
\end{align*}
the real part of the Fourier multiplier is
\begin{align*}
\displaystyle
\operatorname{Re}\hat{k}(\omega)
=
\int_{[0,\infty)}
\frac{\lambda}{\lambda^2+\omega^2}\,\diff\nu(\lambda).
\end{align*}
This quantity measures the effective frequency-by-frequency coercivity supplied by the memory and motivates the following definition.

\begin{definition}[Memory coercivity symbol]\label{def:symbol}
The memory coercivity symbol of a completely monotone kernel $k$ with representing measure $\nu$ is
\begin{align}\label{eq:symbol}
m(\omega)
:=
\int_{[0,\infty)}
\frac{\lambda}{\lambda^2+\omega^2}\,\diff\nu(\lambda),
\quad
\omega\in\mathbb{R}.
\end{align}
When $k\in L^{1}(0,\infty)$, the causal Fourier transform $\hat k$ is well defined and
\begin{align*}
\displaystyle
m(\omega)=\operatorname{Re}\hat{k}(\omega).
\end{align*}
For $\omega\neq0$, the integral converges: the integrand is bounded on $[0,1]$ and dominated by $\lambda^{-1}$ on $[1,\infty)$, and $\int_{[1,\infty)}\lambda^{-1}\diff\nu(\lambda)<\infty$ for any locally integrable kernel. At $\omega=0$, the integrand is read as $\lambda^{-1}$, with the convention $0^{-1}:=+\infty$. Thus,
\begin{align*}
\displaystyle
m(0)
=
\int_{[0,\infty)}
\lambda^{-1}\diff\nu(\lambda)
\in(0,\infty],
\end{align*}
and, in particular, $m(0)=+\infty$ whenever $\nu(\{0\})>0$.
Indeed, when $\norm{k}_{L^{1}(0,\infty)}<\infty$, the integral defining $\hat k$ converges absolutely, the interchange of integrals above is justified by Fubini's theorem, and the identity holds exactly. For kernels of infinite $L^{1}$-mass, $\hat k$ need not exist as a Lebesgue integral, and \cref{eq:symbol} is taken as the definition of the symbol. The computation above is then a formal motivation. This identification is never used beyond the finite-mass setting: \cref{prop:freq} and \cref{thm:gap} assume $\norm{k}_{L^{1}(0,\infty)}<\infty$, whereas the finite-horizon dichotomy \cref{prop:dichotomy,thm:nogo} uses only local $L^{1}$ integrability and positive type.
This is the classical positive-definiteness quantity for Volterra kernels \cite{NohelShea1976,MacCamyWong1972,GripenbergLondenStaffans1990}.
\end{definition}

The following lemma introduces the basic qualitative behaviour of the symbol. Each relaxation mode contributes the factor
\begin{align*}
\displaystyle
\frac{\lambda}{\lambda^{2}+\omega^{2}},
\end{align*}
which is largest at zero frequency and decreases as $\lvert\omega\rvert$ increases. Therefore, the memory acts most strongly on slowly varying components. At $\omega=0$ the symbol equals the total mass of the kernel, and away from zero it is finite even when the mass is infinite. It tends to zero at high frequencies, and this is the source of the loss of coercivity studied below.

\begin{lemma}[Monotonicity and decay]\label{lem:symbol}
The symbol $m$ is even, finite for any $\omega\neq0$, non-increasing as a function of $|\omega|$, and satisfies
\begin{align*}
\displaystyle
\sup_{\omega\in\mathbb R}m(\omega)
= m(0)
= \|k\|_{L^1(0,\infty)}
= \int_{[0,\infty)}\lambda^{-1}\,\diff\nu(\lambda) \in(0,\infty].
\end{align*}
Furthermore,
\begin{align*}
\displaystyle
\lim_{|\omega|\to\infty}m(\omega)=0 .
\end{align*}
The decay holds for any locally integrable completely monotone kernel, including kernels of infinite $L^1$-mass.
\end{lemma}

\begin{proof}
We write $q_{\lambda}(\omega):=\lambda/(\lambda^{2}+\omega^{2})$, so that $m(\omega)=\int_{[0,\infty)}q_{\lambda}(\omega)\diff\nu(\lambda)$. Each $q_{\lambda}$ is even and non-increasing in $|\omega|$. Integrating against the non-negative measure $\nu$ shows that $m$ has the same two properties, with $m(0)$ read in the extended sense.

The measure $\nu$ is finite on $[0,1)$. Indeed, fixing $\tau>0$ and using \cref{eq:bernstein}, $e^{-\tau}\nu([0,1])\leq\int_{[0,1]}e^{-\lambda\tau}\diff\nu(\lambda)\leq k(\tau)<\infty$, so $\nu([0,1])<\infty$. Local integrability of $k$ is equivalent to $\int_{[1,\infty)}\lambda^{-1}\diff\nu(\lambda)<\infty$: by \cref{eq:bernstein} and Tonelli's theorem
\begin{align*}
\int_{0}^{\Tend}k(t)\diff t=\Tend\,\nu(\{0\})+\int_{(0,\infty)}\frac{1-e^{-\lambda\Tend}}{\lambda}\diff\nu(\lambda),
\end{align*}
whose finiteness, because $(1-e^{-\lambda\Tend})/\lambda\leq\Tend$ on $[0,1)$ and $(1-e^{-\Tend})/\lambda\leq(1-e^{-\lambda\Tend})/\lambda\leq\lambda^{-1}$ on $[1,\infty)$, is controlled by $\int_{[1,\infty)}\lambda^{-1}\diff\nu$.

\emph{Finiteness for $\omega\neq0$.} Splitting at $\lambda=1$ and using $q_{\lambda}(\omega)\leq\omega^{-2}$ on $[0,1)$ and $q_{\lambda}(\omega)\leq\lambda^{-1}$ on $[1,\infty)$,
\begin{align*}
m(\omega)\leq\frac{\nu([0,1))}{\omega^{2}}+\int_{[1,\infty)}\lambda^{-1}\diff\nu(\lambda)<\infty,\quad\omega\neq0.
\end{align*}

\emph{Value at zero.} Setting $\omega=0$ and using \cref{eq:bernstein} with Tonelli's theorem,
\begin{align*}
m(0)=\int_{[0,\infty)}\lambda^{-1}\diff\nu(\lambda)=\int_{0}^{\infty}k(t)\diff t=\norm{k}_{L^{1}(0,\infty)}\in(0,\infty],
\end{align*}
where $\lambda^{-1}=+\infty$ at $\lambda=0$. By monotonicity, $\sup_{\omega\in\R}m(\omega)=m(0)$.

\emph{High-frequency decay.} Let $\varepsilon>0$ and choose $R\geq1$ with $\int_{[R,\infty)}\lambda^{-1}\diff\nu(\lambda)<\varepsilon$. Because $q_{\lambda}(\omega)\leq\lambda^{-1}$, the tail satisfies $\int_{[R,\infty)}q_{\lambda}(\omega)\diff\nu<\varepsilon$ for any $\omega$. On the finite measure $\nu|_{[0,R)}$ we have $q_{\lambda}(\omega)\to0$ pointwise as $|\omega|\to\infty$ and $q_{\lambda}(\omega)\leq R$ for $|\omega|\geq1$, so dominated convergence gives $\int_{[0,R)}q_{\lambda}(\omega)\diff\nu\to0$. Therefore, $\limsup_{|\omega|\to\infty}m(\omega)\leq\varepsilon$ for any $\varepsilon>0$. Since $m\geq0$, we conclude $\lim_{|\omega|\to\infty}m(\omega)=0$.
\end{proof}

\begin{example}[The fractional symbol]\label{ex:fracsymbol}
For the fractional kernel of \cref{ex:kernels}, the symbol has the following closed form. With $\diff\nu_{\alpha}(\lambda)=\tfrac{\sin(\pi\alpha)}{\pi}\lambda^{\alpha-1}\diff\lambda$ and the substitution $\lambda=|\omega|s$, for any $\omega\neq0$,
\begin{align*}
\displaystyle
m(\omega)
=\frac{\sin(\pi\alpha)}{\pi}
 \int_{0}^{\infty}\frac{\lambda^{\alpha}}{\lambda^{2}+\omega^{2}}\diff\lambda
=\frac{\sin(\pi\alpha)}{\pi}\,|\omega|^{\alpha-1}
 \int_{0}^{\infty}\frac{s^{\alpha}}{s^{2}+1}\diff s
=\sin \left (\frac{\pi\alpha}{2} \right)\,|\omega|^{\alpha-1},
\end{align*}
where $\int_{0}^{\infty}s^{\alpha}(1+s^{2})^{-1}\diff s =\tfrac{\pi}{2} \left(\cos\tfrac{\pi\alpha}{2}\right)^{-1}$ for $\alpha\in(0,1)$ and $\sin(\pi\alpha)=2\sin(\tfrac{\pi\alpha}{2})\cos(\tfrac{\pi\alpha}{2})$ were used. Thus,
\begin{align*}
\displaystyle
m(\omega)
=
c_{\alpha}|\omega|^{-(1-\alpha)},
\quad
c_{\alpha}
=
\sin \left(\frac{\pi\alpha}{2} \right)>0.
\end{align*}
This example illustrates two points. First, the high-frequency decay $m(\omega)\to0$ in \cref{lem:symbol} remains valid even though
\begin{align*}
\displaystyle
m(0)=\norm{k}_{L^{1}(0,\infty)}=\infty.
\end{align*}
Second, the decay rate is explicitly $1-\alpha$, which will later give the memory-symbol decay index
\begin{align*}
\displaystyle
\rho(k)=1-\alpha.
\end{align*}
Therefore, the fractional memory is not frequency-uniformly coercive, but the precise rate at which its coercive strength is lost can be quantified.
\end{example}

\subsection{Frequency representation and the exact gap}
For a real Hilbert space $X$, we extend $W\in L^2(0,\Tend;X)$ by zero to $\R$ and use the Fourier convention
\begin{align*}
\widehat W(\omega):=\int_{\R}W(t)e^{-i\omega t}\diff t.
\end{align*}
The transform takes values in the complexification $X_{\C}$ and Plancherel's theorem reads
\begin{align*}
\int_{\R}\|W(t)\|_X^2\diff t
=\frac1{2\pi}\int_{\R}\|\widehat W(\omega)\|_{X_{\C}}^2\diff\omega.
\end{align*}

\begin{proposition}[Zero-extension frequency identity]\label{prop:freq}
Let $k\in L^1(0,\infty)$ be real-valued and let $W\in L^2(0,\Tend;X)$ be real Hilbert-valued. Then,
\begin{align}\label{eq:freqrep-general}
\int_0^{\Tend}((k*W)(t),W(t))_X\diff t
=\frac1{2\pi}\int_{\R}\Re\widehat k(\omega)\,
\|\widehat W(\omega)\|_{X_{\C}}^2\diff\omega.
\end{align}
If $k$ is completely monotone with finite $L^1$-mass, then $\Re\widehat k=m$, and for the diffusion realisation, every $u\in\mathcal A_{\Tend}$ satisfies
\begin{align}\label{eq:freqrep}
\Dm[u](\Tend)
=\frac1{2\pi}\int_{\R}m(\omega)\,
\norm{\widehat W(\omega)}_{\LpC}^{2}\diff\omega.
\end{align}
\end{proposition}

\begin{proof}
Extend $k$ causally by zero to $\R$ and $W$ by zero outside $(0,\Tend)$. For $t\in(0,\Tend)$ the extended convolution coincides with $(k*W)(t)=\int_0^tk(t-s)W(s)\diff s$, while the extended $W$ vanishes outside $(0,\Tend)$; hence the left-hand side of \cref{eq:freqrep-general} equals $\int_{\R}((k*W)(t),W(t))_X\diff t$. Young's inequality gives $k*W\in L^2(\R;X)$, and the convolution theorem followed by Plancherel yields
\begin{align*}
\int_{\R}((k*W)(t),W(t))_X\diff t
=\frac1{2\pi}\int_{\R}\widehat k(\omega)\|\widehat W(\omega)\|_{X_{\C}}^2\diff\omega.
\end{align*}
The left-hand side is real. Because $k$ and $W$ are real-valued, $\Im\widehat k$ is odd whereas $\|\widehat W\|^2$ is even, so the imaginary part integrates to zero. This is the standard zero-extension Fourier identity for Volterra convolution; compare \cite[Chap.~16]{GripenbergLondenStaffans1990}. For a completely monotone finite-mass kernel, Fubini in the Bernstein representation gives $\Re\widehat k(\omega)=m(\omega)$.
\end{proof}

\begin{proposition}[Fractional frequency identity]\label{prop:fractional-frequency}
Let
\begin{align*}
k_{\alpha}(t):=\frac{t^{-\alpha}}{\Gamma(1-\alpha)},
\quad 0<\alpha<1.
\end{align*}
Then, for every $u\in\mathcal A_{\Tend}$, with $W=Bu$ extended by zero to $\R$,
\begin{align}\label{eq:fractional-frequency}
\Dm[u](\Tend)
=
\frac{\sin(\pi\alpha/2)}{2\pi}
\int_{\R}
|\omega|^{-(1-\alpha)}
\norm{\widehat W(\omega)}_{\LpC}^{2}
\diff\omega.
\end{align}
The integral on the right-hand side is finite. Thus, in the fractional case the memory dissipation is exactly the homogeneous negative-order Sobolev quadratic form of order $(1-\alpha)/2$ on the zero extension of $W$.
\end{proposition}

\begin{proof}
For $\varepsilon>0$, we set
\begin{align*}
k_{\alpha,\varepsilon}(t):=e^{-\varepsilon t}k_{\alpha}(t).
\end{align*}
Then, $k_{\alpha,\varepsilon}\in L^{1}(0,\infty)$, and the Gamma integral gives
\begin{align*}
\widehat{k_{\alpha,\varepsilon}}(\omega)
=
(\varepsilon+i\omega)^{\alpha-1},
\end{align*}
with the principal branch. Therefore, by \cref{prop:freq},
\begin{align}\label{eq:fractional-eps}
\mathcal{D}_{\mu_{\varepsilon}}[u](\Tend)
=
\frac{1}{2\pi}
\int_{\R}
\Re\bigl((\varepsilon+i\omega)^{\alpha-1}\bigr)
\norm{\widehat W(\omega)}_{\LpC}^{2}
\diff\omega,
\end{align}
where $\mathcal{D}_{\mu_{\varepsilon}}$ denotes the quadratic form generated by $k_{\alpha,\varepsilon}$.

Because $k_{\alpha}\in L^{1}(0,\Tend)$ and $0\leq k_{\alpha}-k_{\alpha,\varepsilon}\leq k_{\alpha}$ on $(0,\Tend)$, dominated convergence gives
\begin{align*}
\norm{k_{\alpha,\varepsilon}-k_{\alpha}}_{L^{1}(0,\Tend)}\longrightarrow0.
\end{align*}
Young's inequality therefore yields
\begin{align*}
|\mathcal{D}_{\mu_{\varepsilon}}[u](\Tend)-\Dm[u](\Tend)|
\leq
\norm{k_{\alpha,\varepsilon}-k_{\alpha}}_{L^{1}(0,\Tend)}
\norm{W}_{L^{2}(0,\Tend;\Lp)}^{2}
\longrightarrow0.
\end{align*}

For any $\omega\neq0$,
\begin{align*}
\Re\bigl((\varepsilon+i\omega)^{\alpha-1}\bigr)
\longrightarrow
\sin(\pi\alpha/2)|\omega|^{\alpha-1}
\quad\text{as }\varepsilon\downarrow0.
\end{align*}
Furthermore, because $\alpha-1<0$,
\begin{align*}
0\leq
\Re\bigl((\varepsilon+i\omega)^{\alpha-1}\bigr)
\leq
|\varepsilon+i\omega|^{\alpha-1}
\leq
|\omega|^{\alpha-1}
\quad (\omega\neq0).
\end{align*}
The zero extension of $W$ has compact support and belongs to $L^{2}(\R;\Lp)$, hence also to $L^{1}(\R;\Lp)$; therefore $\widehat W$ is bounded. It follows that
\begin{align*}
|\omega|^{\alpha-1}\norm{\widehat W(\omega)}_{\LpC}^{2}
\end{align*}
is integrable near $\omega=0$, because $\alpha-1>-1$, and is integrable for $|\omega|\geq1$ by Plancherel's theorem, because $|\omega|^{\alpha-1}\leq1$ there. Dominated convergence in \cref{eq:fractional-eps} now gives \cref{eq:fractional-frequency}.
\end{proof}

The frequency identity gives the following classical exact-gap formula. For a real-valued $W\in L^2(\R;\Lp)$, we define
\begin{align}\label{eq:wholegap}
\mathcal G_{\R}(W):=\frac1{2\pi}\int_{\R}(1-m(\omega))
\norm{\widehat W(\omega)}_{\LpC}^{2}\diff\omega.
\end{align}

\begin{proposition}[Exact coercivity gap and sign criterion]\label{thm:gap}
Let $k$ be a locally integrable completely monotone kernel and assume $\norm{k}_{L^{1}(0,\infty)}<\infty$. Then, for every $u\in\mathcal A_{\Tend}$,
\begin{align}\label{eq:gap}
\int_{0}^{\Tend}a_{1}(u(t),u(t))\diff t-\Dm[u](\Tend)
=\frac{1}{2\pi}\int_{\R}(1-m(\omega))
\norm{\widehat W(\omega)}_{\LpC}^{2}\diff\omega.
\end{align}
Furthermore,
\begin{enumerate}[leftmargin=2.2em,label=\textnormal{(\roman*)}]
\item If $\norm{k}_{L^{1}(0,\infty)}\leq1$, then the gap in \cref{eq:gap} is non-negative for every $\Tend>0$ and every $u\in\mathcal A_{\Tend}$.
\item Suppose, in addition, that $a_1$ is not identically zero. If $\norm{k}_{L^{1}(0,\infty)}>1$, then $\mathcal G_{\R}(W)<0$ for some real-valued $W\in L^2(\R;\Lp)$. The corresponding finite-horizon negative gap follows from \cref{thm:finite-profile}\textnormal{(iii), (v)} and \cref{cor:critical-horizon}.
\end{enumerate}
Consequently, when $a_1\not\equiv0$, the gap is non-negative for every $\Tend>0$ and every $u\in\mathcal A_{\Tend}$ if and only if $\norm{k}_{L^{1}(0,\infty)}\leq1$.
\end{proposition}

\begin{proof}
Equation \cref{eq:gap} is \cref{prop:freq} minus Plancherel. Because $0\leq m(\omega)\leq m(0)=\|k\|_{L^1}$, part~\textnormal{(i)} follows. If $m(0)>1$, continuity of $m=\Re\widehat k$ gives an interval $(-\delta,\delta)$ on which $m>1$. Choose $w_0\in V$ with $a_1(w_0,w_0)>0$ and a non-zero real even $\widehat g\in C_c^\infty(-\delta,\delta)$; then $W(t):=g(t)Bw_0$ satisfies $\mathcal G_{\R}(W)<0$. The finite-horizon assertion follows later from the profile criterion, without using this proposition.
\end{proof}

We include this proposition for completeness; it is a direct consequence of Plancherel's theorem.

\subsection{Cosine modes and the coercivity profile}
\label{subsec:finite-horizon}
\Cref{prop:freq,thm:gap} describe the memory through $m$ after extension to the whole line. We now show that $m$ can also be read off on a bounded interval, from the Rayleigh quotients of explicit modes.

\begin{definition}[Finite-horizon quadratic form and coercivity profile]
\label{def:finite-profile}
Let $\Tend>0$ and let $k\in L^1(0,\Tend)$ be real-valued. For $\phi\in L^{2}(0,\Tend;\R)$, we define
\begin{align}
\mathcal{Q}_{k,\Tend}[\phi]
:=
\int_{0}^{\Tend}\phi(t)\int_{0}^{t}k(t-s)\phi(s)\diff s\diff t,
\label{eq:finite-quadratic}
\end{align}
and
\begin{align}
\Lambda_{k}(\Tend)
:=
\sup_{\substack{\phi\in L^2(0,\Tend)\\\phi\neq0}}
\frac{\mathcal{Q}_{k,\Tend}[\phi]}
{\|\phi\|_{L^{2}(0,\Tend)}^{2}}.
\label{eq:finite-profile}
\end{align}
The associated symmetric operator is
\begin{align}
(\mathcal S_{k,\Tend}\phi)(t)
:=\frac12\int_{0}^{\Tend}k(|t-s|)\phi(s)\diff s,
\label{eq:finite-operator}
\end{align}
so that
\begin{align}
\mathcal Q_{k,\Tend}[\phi]
=(\mathcal S_{k,\Tend}\phi,\phi)_{L^{2}(0,\Tend)}.
\label{eq:finite-form-operator}
\end{align}
\end{definition}

For $n\in\mathbb N$, we set
\begin{align}
\omega_n:=\frac{2\pi n}{\Tend},
\quad
\psi_n(t):=\sqrt{\frac{2}{\Tend}}\cos(\omega_nt),
\quad
\|\psi_n\|_{L^{2}(0,\Tend)}=1.
\label{eq:cosine-sequence}
\end{align}

\begin{theorem}[Finite-horizon realisation of the memory symbol]
\label{thm:finite-symbol}
Let $k$ be a locally integrable completely monotone kernel with Bernstein representing measure $\nu$, and let $\Tend>0$. Then,
\begin{align}
\mathcal Q_{k,\Tend}[\psi_n]
={}&
\int_{[0,\infty)}
\left\{
\frac{\lambda}{\lambda^2+\omega_n^2}
-
\frac{2\lambda^2(1-e^{-\lambda\Tend})}
{\Tend(\lambda^2+\omega_n^2)^2}
\right\}\diff\nu(\lambda).
\label{eq:finite-symbol-exact}
\end{align}
We set $a_n:=2\pi n$ and
\begin{align}
\kappa_n
:=
\max_{x\geq0}
\frac{2a_nx(1-e^{-x})}{x^2+a_n^2}.
\label{eq:kappa-n}
\end{align}
Then,
\begin{align}
1-e^{-a_n}\leq\kappa_n\leq1,
\label{eq:kappa-bounds}
\end{align}
and the optimal universal comparison is
\begin{align}
\left(1-\frac{\kappa_n}{a_n}\right)m(\omega_n)
\leq
\mathcal Q_{k,\Tend}[\psi_n]
\leq
m(\omega_n).
\label{eq:finite-symbol-sharp}
\end{align}
In particular,
\begin{align}
\left(1-\frac{1}{2\pi n}\right)m(\omega_n)
\leq
\mathcal Q_{k,\Tend}[\psi_n]
\leq
m(\omega_n).
\label{eq:finite-symbol-bounds-n}
\end{align}
If $k$ is non-constant, then
\begin{align}
\frac{\mathcal Q_{k,\Tend}[\psi_n]}{m(\omega_n)}\longrightarrow1
\quad(n\to\infty).
\label{eq:finite-symbol-ratio}
\end{align}
Thus, the full-line memory symbol is asymptotically realised by finite-horizon Rayleigh quotients with a universal relative $O(n^{-1})$ error.
\end{theorem}

\begin{proof}
By Bernstein's representation and absolute Fubini,
\begin{align*}
\mathcal Q_{k,\Tend}[\psi_n]
=
\int_{[0,\infty)}q_{\lambda,n}\diff\nu(\lambda),
\end{align*}
where
\begin{align*}
q_{\lambda,n}
:=
\int_0^{\Tend}\psi_n(t)\int_0^t e^{-\lambda(t-s)}\psi_n(s)\diff s\diff t.
\end{align*}
The interchange is justified by $k\in L^{1}(0,\Tend)$ and boundedness of $\psi_n$. With $r=t-s$ and $\omega_n\Tend=2\pi n$, the product-to-sum identity gives
\begin{align*}
q_{\lambda,n}
=
\frac1{\Tend}\int_0^{\Tend}e^{-\lambda r}
\left[(\Tend-r)\cos(\omega_nr)-\omega_n^{-1}\sin(\omega_nr)\right]\diff r,
\end{align*}
and elementary integration yields
\begin{align*}
q_{\lambda,n}
=
\frac{\lambda}{\lambda^2+\omega_n^2}
\left[1-
\frac{2\lambda(1-e^{-\lambda\Tend})}
{\Tend(\lambda^2+\omega_n^2)}\right].
\end{align*}
This proves \cref{eq:finite-symbol-exact}. We set $x := \lambda\Tend$ and $a_n=\omega_n\Tend=2\pi n$. Then, the relative correction equals
\begin{align*}
d_n(x):=\frac{2x(1-e^{-x})}{x^2+a_n^2}.
\end{align*}
The function $a_nd_n(x)$ is continuous on $[0,\infty)$ and tends to zero at both endpoints, so the maximum in \cref{eq:kappa-n} is attained. Because $2a_nx\leq x^2+a_n^2$,
\begin{align*}
a_nd_n(x)\leq1,
\end{align*}
whereas evaluation at $x=a_n$ gives $a_nd_n(a_n)=1-e^{-a_n}$. Therefore, \cref{eq:kappa-bounds}. The pointwise estimate
\begin{align*}
1-\frac{\kappa_n}{a_n}\leq1-d_n(x)\leq1
\end{align*}
multiplied by $\lambda/(\lambda^2+\omega_n^2)$ and integrated against $\nu$ gives \cref{eq:finite-symbol-sharp}. The simpler estimate \cref{eq:finite-symbol-bounds-n} follows from $\kappa_n\leq1$.

The lower coefficient in \cref{eq:finite-symbol-sharp} is optimal uniformly over completely monotone kernels. Indeed, let $x_n^{\ast}>0$ be a maximiser in \cref{eq:kappa-n} and take the single-exponential kernel $k(t)=c e^{-\lambda t}$ with $c>0$ and $\lambda=x_n^{\ast}/\Tend$. Its representing measure is $c\delta_{\lambda}$, and the ratio $\mathcal Q_{k,\Tend}[\psi_n]/m(\omega_n)$ is exactly $1-\kappa_n/a_n$. The upper coefficient $1$ is also optimal, since the same ratio tends to $1$ along single-exponential kernels with $\lambda\Tend\downarrow0$. Finally, if $k$ is non-constant, then $m(\omega_n)>0$, and \cref{eq:kappa-bounds,eq:finite-symbol-sharp} imply \cref{eq:finite-symbol-ratio}.
\end{proof}

The modes $\psi_n$ are also the weakly null sequence used below to show that the coercivity constant vanishes, so \cref{thm:finite-symbol} gives the rate at which this happens on $(0,\Tend)$ itself.

\begin{proposition}[Distributed--instantaneous coercivity floor]
\label{prop:dichotomy}
Let $\Tend>0$ and let $k\in L^{1}(0,\Tend)$ be of positive type on $(0,\Tend)$ in the scalar sense. Then, $\mathcal S_{k,\Tend}$ is compact, self-adjoint and non-negative on $L^{2}(0,\Tend)$, and
\begin{align}
\inf_{\substack{\phi\in L^2(0,\Tend)\\\phi\neq0}}
\frac{\mathcal Q_{k,\Tend}[\phi]}{\|\phi\|_{L^{2}(0,\Tend)}^{2}}
=0.
\label{eq:distributed-floor}
\end{align}
More precisely, $\mathcal Q_{k,\Tend}[\psi_n]\to0$. If $X$ is a non-zero real Hilbert space and $a\geq0$, then
\begin{align}
\mathcal Q^{X}_{a,k,\Tend}[W]
:=
a\|W\|_{L^{2}(0,\Tend;X)}^{2}
+
\int_{0}^{\Tend}((k*W)(t),W(t))_{X}\diff t
\label{eq:atomic-distributed-form}
\end{align}
satisfies
\begin{align}
\inf_{\substack{W\in L^2(0,\Tend;X)\\W\neq0}}
\frac{\mathcal Q^{X}_{a,k,\Tend}[W]}
{\|W\|_{L^{2}(0,\Tend;X)}^{2}}
=a.
\label{eq:atomic-distributed-floor}
\end{align}
Thus, the instantaneous coefficient is exactly the uniform $L^{2}$ coercivity floor, while the distributed positive-type part contributes zero to that floor.
\end{proposition}

\begin{proof}
The kernel $\tfrac12k(|t-s|)$ is symmetric. Young's inequality gives boundedness; approximation of $k(|\cdot|)$ in $L^1(-\Tend,\Tend)$ by smooth kernels makes $\mathcal S_{k,\Tend}$ an operator-norm limit of Hilbert--Schmidt operators, hence compact. Positive type gives non-negativity. Because $(\psi_n)$ is orthonormal and weakly null, compactness gives $\mathcal S_{k,\Tend}\psi_n\to0$, proving \cref{eq:distributed-floor}. For Hilbert-valued fields, positivity follows by finite-dimensional expansion and density. Thus, the quotient in \cref{eq:atomic-distributed-floor} is at least $a$, while $W_n(t)=\psi_n(t)e$ for a fixed unit vector $e\in X$ gives the reverse inequality.
\end{proof}

\begin{corollary}[Finite-horizon no-go for distributed memory]
\label{thm:nogo}
Let $\Tend>0$, let $k$ be a locally integrable completely monotone kernel, and suppose that $a_1\not\equiv0$. Then,
\begin{align}
\inf_{u\in\mathcal A_{\Tend}^{+}}
\frac{\Dm[u](\Tend)}{\displaystyle\int_0^{\Tend}a_1(u(t),u(t))\diff t}=0.
\label{eq:nogo-inf}
\end{align}
 Consequently, no $c>0$ can make the memory dissipation dominate the instantaneous $a_1$-energy uniformly. Furthermore, without any coercivity assumption on $a_1$, there is no $c>0$ such that
\begin{align}
\Dm[u](\Tend)\geq c\|u\|_{L^{2}(0,\Tend;V)}^{2}
\label{eq:nogo-V}
\end{align}
for any $u\in L^{2}(0,\Tend;V)$.
\end{corollary}

\begin{proof}
By \cref{lem:posrep}, $k$ is of positive type. Choose $w_0\in V$ with $c_0:=a_1(w_0,w_0)>0$ and set $u_n(t):=\psi_n(t)w_0$. Then,
\begin{align*}
\int_0^{\Tend}a_1(u_n,u_n)\diff t=c_0,
\quad
\Dm[u_n](\Tend)=c_0\mathcal Q_{k,\Tend}[\psi_n]\longrightarrow0
\end{align*}
by \cref{prop:dichotomy}; this proves \cref{eq:nogo-inf}. Since $\|u_n\|_{L^{2}(0,\Tend;V)}^{2}=\|w_0\|_{V}^{2}$, the same sequence proves \cref{eq:nogo-V}.
\end{proof}

Here, compactness is needed only for the scalar temporal operator. The Hilbert-valued floor is obtained by tensoring the scalar near-null sequence with one fixed unit vector.

\begin{theorem}[Finite-horizon coercivity profile]
\label{thm:finite-profile}
Let $k$ be a non-trivial locally integrable completely monotone kernel. Then, the following assertions hold.

\begin{enumerate}[leftmargin=2.2em,label=\textnormal{(\roman*)}]
\item
For any $\Tend>0$, the operator $\mathcal{S}_{k,\Tend}$ is bounded, compact, self-adjoint, and non-negative on $L^{2}(0,\Tend)$. Furthermore,
\begin{align}
\Lambda_k(\Tend)
=
\norm{\mathcal{S}_{k,\Tend}}_
{\mathcal{L}(L^{2}(0,\Tend))}
\label{eq:profile-operator-norm}
\end{align}
and
\begin{align}
0
<
\Lambda_k(\Tend)
\leq
\norm{k}_{L^{1}(0,\Tend)}.
\label{eq:profile-upper}
\end{align}

\item
The map
\begin{align*}
\Tend\longmapsto\Lambda_k(\Tend)
\end{align*}
is strictly increasing and continuous on $(0,\infty)$.

\item
One has
\begin{align}
\lim_{\Tend\downarrow0}\Lambda_k(\Tend)=0
\label{eq:profile-zero}
\end{align}
and, in the extended sense,
\begin{align}
\lim_{\Tend\to\infty}\Lambda_k(\Tend)
= 
\norm{k}_{L^{1}(0,\infty)}
\in(0,\infty].
\label{eq:profile-infinity}
\end{align}

\item
Let $X$ be a non-zero real Hilbert space. Then, for any $W\in L^{2}(0,\Tend;X)$,
\begin{align}
0
\leq
\int_{0}^{\Tend}
\left((k*W)(t),W(t)\right)_{X}\diff t
\leq
\Lambda_k(\Tend)
\norm{W}_{L^{2}(0,\Tend;X)}^{2}.
\label{eq:vector-profile}
\end{align}
The constant $\Lambda_k(\Tend)$ is optimal.

\item
Suppose that $a_1\not\equiv0$. Then,
\begin{align}
\int_{0}^{\Tend}
a_1\left(u(t),u(t)\right)\diff t
-
\Dm[u](\Tend)
\geq0
\quad\text{for every }u\in\mathcal A_{\Tend}
\label{eq:fixed-gap-positive}
\end{align}
if and only if
\begin{align}
\Lambda_k(\Tend)\leq1.
\label{eq:fixed-gap-criterion}
\end{align}
\end{enumerate}
\end{theorem}

\begin{proof}
Part~\textnormal{(i)} follows from \cref{prop:dichotomy}: $\mathcal S_{k,\Tend}$ is compact, self-adjoint and non-negative, hence $\Lambda_k(\Tend)=\|\mathcal S_{k,\Tend}\|$. Young's inequality gives the upper bound in \cref{eq:profile-upper}, while $k(r)>0$ for $r>0$ and the test function $1$ give $\Lambda_k(\Tend)>0$.

For $0<\Tend_{1}<\Tend_{2}$ and $\phi\in L^{2}(0,\Tend_{1})$, its zero-extension $\widetilde\phi\in L^{2}(0,\Tend_{2})$ satisfies $\norm{\widetilde\phi}_{L^{2}(0,\Tend_{2})}=\norm{\phi}_{L^{2}(0,\Tend_{1})}$ and, since the outer factor vanishes for $t\geq\Tend_{1}$, $\mathcal{Q}_{k,\Tend_{2}}[\widetilde\phi]=\mathcal{Q}_{k,\Tend_{1}}[\phi]$. Taking the supremum gives $\Lambda_{k}(\Tend_{1})\leq\Lambda_{k}(\Tend_{2})$. We now show that the inequality is strict. Because $\mathcal S_{k,\Tend_{1}}$ is compact, self-adjoint and non-negative and $\Lambda_k(\Tend_1)>0$, the spectral theorem gives a unit vector $\phi_1\in L^{2}(0,\Tend_1)$ with
\begin{align*}
\mathcal S_{k,\Tend_1}\phi_1=\Lambda_k(\Tend_1)\phi_1.
\end{align*}
Because the kernel $k(|t-s|)/2$ is strictly positive for $s\neq t$ (a non-trivial completely monotone kernel satisfies $k(r)>0$ for any $r>0$), replacing $\phi_1$ by $|\phi_1|$ does not decrease its Rayleigh quotient. We may therefore take $\phi_1\geq0$ almost everywhere. Let $\eta\in L^{2}(0,\Tend_2)$ be non-zero, non-negative and supported in $(\Tend_1,\Tend_2)$, and normalise it so that $\norm{\eta}_{L^{2}(0,\Tend_2)}=1$. If $\widetilde\phi_1$ denotes the zero-extension of $\phi_1$, then
\begin{align*}
B:=\left(\mathcal S_{k,\Tend_2}\widetilde\phi_1,\eta\right)_{L^{2}(0,\Tend_2)}>0.
\end{align*}
Indeed, the double-integral representation has a strictly positive kernel between the two disjoint supports. Therefore, for $\varepsilon>0$,
\begin{align*}
&\mathcal Q_{k,\Tend_2}[\widetilde\phi_1+\varepsilon\eta]-\Lambda_k(\Tend_1)\norm{\widetilde\phi_1+\varepsilon\eta}_{L^{2}(0,\Tend_2)}^{2}
=2\varepsilon B+\varepsilon^{2}\left(\mathcal Q_{k,\Tend_2}[\eta]-\Lambda_k(\Tend_1)\right)>0
\end{align*}
for all sufficiently small $\varepsilon$. Therefore, $\Lambda_k(\Tend_2)>\Lambda_k(\Tend_1)$.

We next prove continuity. Fix $\Tend_{0}>0$ and choose $R>\Tend_{0}$. Let $\mathcal H_R:=L^{2}(0,R)$ and let $P_{\Tend}$ denote multiplication by $\chi_{(0,\Tend)}$ on $\mathcal H_R$ for $0<\Tend<R$. We define the bounded self-adjoint operator
\begin{align*}
(\mathcal S_{k,R}\phi)(t):=\frac12\int_{0}^{R}k(|t-s|)\phi(s)\diff s.
\end{align*}
By \cref{prop:dichotomy} applied on $(0,R)$, the operator $\mathcal S_{k,R}$ is compact on $\mathcal H_R$.

After identifying $L^{2}(0,\Tend)$ with $P_{\Tend}\mathcal H_R$ by zero extension, $\mathcal S_{k,\Tend}$ is exactly the restriction of $P_{\Tend}\mathcal S_{k,R}P_{\Tend}$ to $P_{\Tend}\mathcal H_R$. Consequently,
\begin{align}\label{eq:profile-compression}
\Lambda_k(\Tend)=\norm{P_{\Tend}\mathcal S_{k,R}P_{\Tend}}_{\mathcal L(\mathcal H_R)}.
\end{align}
As $\Tend\to\Tend_0$, the projections $P_{\Tend}$ converge strongly to $P_{\Tend_0}$. Because $\mathcal S_{k,R}$ is compact, strong convergence is uniform on the compact image of the unit ball, and therefore
\begin{align*}
\norm{(P_{\Tend}-P_{\Tend_0})\mathcal S_{k,R}}_{\mathcal L(\mathcal H_R)}\longrightarrow0.
\end{align*}
By self-adjointness, the same holds for $\mathcal S_{k,R}(P_{\Tend}-P_{\Tend_0})$. Therefore,
\begin{align*}
&\norm{P_{\Tend}\mathcal S_{k,R}P_{\Tend}-P_{\Tend_0}\mathcal S_{k,R}P_{\Tend_0}}_{\mathcal L(\mathcal H_R)}\\
&\quad\leq
\norm{(P_{\Tend}-P_{\Tend_0})\mathcal S_{k,R}P_{\Tend}}+
\norm{P_{\Tend_0}\mathcal S_{k,R}(P_{\Tend}-P_{\Tend_0})}
\longrightarrow0.
\end{align*}
Taking norms in \cref{eq:profile-compression} gives $\Lambda_k(\Tend)\to\Lambda_k(\Tend_0)$. Thus, $\Lambda_k$ is continuous on $(0,\infty)$, proving \textnormal{(ii)}.

For \textnormal{(iii)}, the upper bound in \cref{eq:profile-upper} gives $\Lambda_k(\Tend)\to0$ as $\Tend\downarrow0$. Testing with the constant function yields
\begin{align*}
\Lambda_k(\Tend)\geq\int_0^\Tend(1-r/\Tend)k(r)\diff r\to\|k\|_{L^1(0,\infty)},
\end{align*}
while \cref{eq:profile-upper} gives the matching upper bound.

For \textnormal{(iv)}, expand finite-rank $X$-valued fields in an orthonormal basis and apply the scalar estimate componentwise; density and Young's inequality pass to general $W$. Optimality follows from $W(t)=\phi(t)e$ with a fixed unit vector $e\in X$. Finally, $W=Bu$, so \cref{eq:vector-profile} gives the non-negative gap when $\Lambda_k(\Tend)\leq1$. If $\Lambda_k(\Tend)>1$, choose $\phi$ with $\mathcal Q_{k,\Tend}[\phi]>\|\phi\|^2$ and $w_0\in V$ with $a_1(w_0,w_0)>0$; the state $u(t)=\phi(t)w_0$ has negative gap. This proves \textnormal{(v)}.
\end{proof}

\begin{corollary}[Critical horizon]
\label{cor:critical-horizon}
Let $k$ be a non-trivial locally integrable completely monotone kernel. Assume that $a_1\not\equiv0$ and
\begin{align*}
\norm{k}_{L^{1}(0,\infty)}>1,
\end{align*}
where the left-hand side may be infinite. We define
\begin{align}
\Tend_{\ast}
:=
\sup
\left\{
\Tend>0:
\Lambda_k(\Tend)\leq1
\right\}.
\label{eq:critical-horizon}
\end{align}
Then,
\begin{align}
0<\Tend_{\ast}<\infty.
\label{eq:critical-finite}
\end{align}
Furthermore, $\Tend_{\ast}$ is the unique positive solution of $\Lambda_k(\Tend)=1$, and
\begin{align}
\Lambda_k(\Tend)<1
&\quad
\text{for }0<\Tend<\Tend_{\ast},
\label{eq:critical-below}
\\
\Lambda_k(\Tend_{\ast})&=1,
\label{eq:critical-equality}
\\
\Lambda_k(\Tend)>1
&\quad
\text{for }\Tend>\Tend_{\ast}.
\label{eq:critical-above}
\end{align}
Consequently, for $0<\Tend\leq\Tend_{\ast}$ the coercivity gap is non-negative for every $u\in\mathcal A_{\Tend}$, whereas for every $\Tend>\Tend_{\ast}$ there exists $u\in\mathcal A_{\Tend}$ with strictly negative gap.
\end{corollary}

\begin{proof}
By \cref{thm:finite-profile}, $\Lambda_k$ is continuous and strictly increasing, tends to $0$ as $\Tend\downarrow0$, and tends to $\|k\|_{L^1}>1$ as $\Tend\to\infty$. The intermediate value theorem therefore gives a unique $\Tend_*>0$ with $\Lambda_k(\Tend_*)=1$, and strict monotonicity gives \cref{eq:critical-below,eq:critical-above}. This unique crossing agrees with the supremum in \cref{eq:critical-horizon}, which proves \cref{eq:critical-finite}. The final gap statement follows from \cref{thm:finite-profile}\textnormal{(v)}.
\end{proof}

\begin{corollary}[Fractional critical horizon]
\label{cor:fractional-horizon}
Let
\begin{align*}
k_{\alpha}(t)
=
\frac{t^{-\alpha}}{\Gamma(1-\alpha)},
\quad
0<\alpha<1.
\end{align*}
Then, for any $\Tend>0$,
\begin{align}
\Lambda_{k_{\alpha}}(\Tend)
=
\Tend^{1-\alpha}
\Lambda_{k_{\alpha}}(1).
\label{eq:fractional-profile-scaling}
\end{align}
In particular,
\begin{align}
\Tend_{\ast,\alpha}
:=
\Lambda_{k_{\alpha}}(1)^{-1/(1-\alpha)}
\in(0,\infty)
\label{eq:fractional-critical-horizon}
\end{align}
is the unique critical horizon satisfying
\begin{align*}
\Lambda_{k_{\alpha}}(\Tend_{\ast,\alpha})=1.
\end{align*}
If $a_1\not\equiv0$, the coercivity gap is non-negative for any $u\in\mathcal A_{\Tend}$ if and only if
\begin{align*}
0<\Tend\leq \Tend_{\ast,\alpha}.
\end{align*}
\end{corollary}

\begin{proof}
The map
\begin{align*}
(U_{\Tend}\psi)(t):=\Tend^{-1/2}\psi(t/\Tend)
\end{align*}
is unitary from $L^2(0,1)$ to $L^2(0,\Tend)$. Because $k_\alpha(\Tend(\tau-\sigma))=\Tend^{-\alpha}k_\alpha(\tau-\sigma)$, the change of variables $t=\Tend\tau$, $s=\Tend\sigma$ gives
\begin{align*}
\mathcal Q_{k_\alpha,\Tend}[U_{\Tend}\psi]=\Tend^{1-\alpha}\mathcal Q_{k_\alpha,1}[\psi].
\end{align*}
Taking Rayleigh suprema proves \cref{eq:fractional-profile-scaling}. Since $0<\Lambda_{k_\alpha}(1)<\infty$ by \cref{thm:finite-profile}, the equation $\Lambda_{k_\alpha}(\Tend)=1$ has the unique solution \cref{eq:fractional-critical-horizon}; the gap criterion is \cref{thm:finite-profile}\textnormal{(v)}.
\end{proof}

Compactness shows that the coercivity constant vanishes for any positive-type kernel in $L^{1}(0,\Tend)$; for completely monotone kernels, \cref{thm:finite-symbol} adds the rate, namely $m(\omega_n)$ up to the factor $1-\kappa_n/(2\pi n)$.

\section{The memory-symbol decay index and the instantaneous limit}
\label{sec:structure}
We show that the comparison of \cref{thm:finite-symbol} determines the decay index of $m$ on every bounded interval, and that this fails for the Gaussian kernel, which is of positive type but not completely monotone. The section ends with the limit in which distributed memory concentrates at $t=0$.

\subsection{The decay index on bounded intervals}

\begin{definition}[Memory-symbol decay indices]\label{def:index}
For a locally integrable completely monotone kernel $k$ with memory symbol $m$, we define
\begin{align*}
\rho(k)
:={}
\sup\left\{
s\geq0:
\limsup_{|\omega|\to\infty}|\omega|^{s}m(\omega)<\infty
\right\}.
\end{align*}
More generally, for any real-valued $k\in L^{1}(0,\infty)$ with $m(\omega):=\operatorname{Re}\widehat{k}(\omega)\geq0$, we use the same definition. By \cref{prop:freq}, this condition makes the associated zero-history quadratic form non-negative on every finite horizon. For $\Tend>0$, we define the finite-horizon cosine decay index as
\begin{align*}
\rho_{\Tend}^{\mathrm{cos}}(k)
:={}
\sup\left\{
s\geq0:
\limsup_{n\to\infty}n^s\mathcal Q_{k,\Tend}[\psi_n]<\infty
\right\},
\end{align*}
whenever the quadratic form is non-negative.
\end{definition}

\begin{theorem}[Finite-horizon rigidity of the decay index]\label{thm:finite-index-rigidity}
Let $k$ be a non-constant locally integrable completely monotone kernel. Then, for any $\Tend>0$,
\begin{align}
\rho_{\Tend}^{\mathrm{cos}}(k)=\rho(k),
\quad
0\leq\rho(k)\leq2.
\label{eq:finite-index-rigidity}
\end{align}
Any value in $[0,2]$ is realised by an absolutely continuous Bernstein measure.

Complete monotonicity cannot be omitted in general. The Gaussian kernel
\begin{align}
k_G(t):=e^{-t^2},\quad t>0,
\label{eq:gaussian-kernel}
\end{align}
is of positive type but is not completely monotone. Its whole-line coercivity symbol is
\begin{align}
m_G(\omega)
=\operatorname{Re}\widehat{k_G}(\omega)
=\frac{\sqrt\pi}{2}e^{-\omega^2/4},
\label{eq:gaussian-symbol}
\end{align}
so that $\rho(k_G)=+\infty$, whereas for any $\Tend>0$,
\begin{align}
\mathcal Q_{k_G,\Tend}[\psi_n]
={}
\frac{4}{\Tend}\bigl[1+(2\Tend^2-1)e^{-\Tend^2}\bigr]\omega_n^{-4}
+O(\omega_n^{-6}),
\quad n\to\infty,
\label{eq:gaussian-finite-asymptotic}
\end{align}
and therefore
\begin{align}
\rho_{\Tend}^{\mathrm{cos}}(k_G)=4\neq\rho(k_G).
\label{eq:gaussian-index-mismatch}
\end{align}
Thus, outside the completely monotone class, the finite-window boundary correction can determine the polynomial decay of the cosine Rayleigh quotients instead of the whole-line symbol.
\end{theorem}

\begin{proof}
For a completely monotone kernel, \cref{eq:finite-symbol-bounds-n} gives
\begin{align*}
\left(1-\frac{1}{2\pi n}\right)m\!\left(\frac{2\pi n}{\Tend}\right)
\leq
\mathcal Q_{k,\Tend}[\psi_n]
\leq
m\!\left(\frac{2\pi n}{\Tend}\right).
\end{align*}
Therefore, the finite-horizon quotient and the sampled symbol have the same polynomial decay index. By \cref{lem:symbol}, $m$ is non-increasing on $(0,\infty)$. Thus, a bound on $\omega^s m(\omega)$ implies the corresponding bound on $n^s m(2\pi n/\Tend)$. Conversely, if $n^s m(2\pi n/\Tend)$ is bounded and $2\pi n/\Tend\leq\omega<2\pi(n+1)/\Tend$, then
\begin{align*}
\omega^s m(\omega)
\leq
\left(\frac{2\pi(n+1)}{\Tend}\right)^s
m\!\left(\frac{2\pi n}{\Tend}\right),
\end{align*}
which is uniformly bounded for large $n$. Thus $\rho_{\Tend}^{\mathrm{cos}}(k)=\rho(k)$.

To obtain the range, we set
\begin{align*}
M_1:=\int_{[0,\infty)}\lambda\,\diff\nu(\lambda)\in(0,\infty].
\end{align*}
Monotone convergence gives
\begin{align}
\omega^2m(\omega)
=
\int_{[0,\infty)}
\frac{\lambda\omega^2}{\lambda^2+\omega^2}\,\diff\nu(\lambda)
\longrightarrow M_1.
\label{eq:omega2m-rigidity}
\end{align}
Because $m(\omega)\to0$, one has $\rho(k)\geq0$. Since $M_1>0$, \cref{eq:omega2m-rigidity} excludes every exponent $s>2$, and hence $\rho(k)\leq2$.

The full interval is realised explicitly. For $0<\beta<2$, we take
\begin{align*}
\diff\nu_\beta(\lambda)=\lambda^{-\beta}\diff\lambda
\quad\text{on }[1,\infty).
\end{align*}
Then, after the substitution $\lambda=|\omega|r$,
\begin{align*}
|\omega|^\beta m(\omega)
=
\int_{1/|\omega|}^{\infty}
\frac{r^{1-\beta}}{1+r^2}\diff r
\longrightarrow
\int_0^\infty\frac{r^{1-\beta}}{1+r^2}\diff r\in(0,\infty),
\end{align*}
so $\rho=\beta$. For $\beta=2$, the same density with exponent $2$ gives
$m(\omega)=\tfrac{1}{2|\omega|^2}\log(1+\omega^2)$ and hence $\rho=2$. Finally,
\begin{align*}
\diff\nu_0(\lambda)=(\log\lambda)^{-2}\diff\lambda
\quad\text{on }[e,\infty)
\end{align*}
gives
\begin{align*}
m(\omega)
\geq\frac12\int_{|\omega|}^{\infty}
\frac{\diff\lambda}{\lambda(\log\lambda)^2}
=\frac{1}{2\log|\omega|}
\end{align*}
for large $|\omega|$, and therefore $\rho=0$. All these representing measures are absolutely continuous and generate locally integrable completely monotone kernels.

It remains to prove the Gaussian statement. Because
\begin{align*}
k_G''(t)=(4t^2-2)e^{-t^2}<0
\quad\text{for }0<t<2^{-1/2},
\end{align*}
the kernel is not completely monotone. On the other hand,
\begin{align*}
\operatorname{Re}\widehat{k_G}(\omega)
=
\int_0^\infty e^{-t^2}\cos(\omega t)\diff t
=
\frac{\sqrt\pi}{2}e^{-\omega^2/4}>0,
\end{align*}
so the general zero-extension identity \cref{prop:freq} shows that $k_G$ is of positive type on every finite horizon. Formula \cref{eq:finite-quadratic} and the product-to-sum identity give, for $\omega_nT=2\pi n$,
\begin{align}
\mathcal Q_{k_G,\Tend}[\psi_n]
={}
\int_0^\Tend e^{-r^2}
\left[
\left(1-\frac rT\right)\cos(\omega_n r)
-
\frac{\sin(\omega_n r)}{\omega_n \Tend}
\right]\diff r.
\label{eq:gaussian-rayleigh}
\end{align}
Repeated integration by parts, using $\sin(\omega_nT)=0$ and $\cos(\omega_nT)=1$, yields
\begin{align*}
\mathcal Q_{k_G,\Tend}[\psi_n]
=
\left[
-\bigl(f'''(\Tend)-f'''(0)\bigr)
-
\frac{k_G''(\Tend)-k_G''(0)}{\Tend}
\right]\omega_n^{-4}
+O(\omega_n^{-6}),
\end{align*}
where $f(r):=(1-r/\Tend)k_G(r)$. Because
\begin{align*}
k_G''(r)=(4r^2-2)e^{-r^2},
\quad
k_G'''(0)=0,
\end{align*}
the coefficient reduces to
\begin{align*}
\frac{4}{\Tend}\left[1+(2\Tend^2-1)e^{-\Tend^2}\right]>0.
\end{align*}
The positivity is immediate for $\Tend^2\geq1/2$; if $0<\Tend^2<1/2$, then $1-(1-2\Tend^2)e^{-\Tend^2}>2\Tend^2>0$. This proves \cref{eq:gaussian-finite-asymptotic}. \cref{eq:gaussian-index-mismatch} follows immediately from the definitions.
\end{proof}

\begin{remark}[Why complete monotonicity matters]\label{rem:cm-rigidity}
For a general positive-type kernel, truncation to $(0,\Tend)$ produces endpoint terms whose decay need not be related to the symbol. In the Gaussian case the symbol is smaller than any power of $\omega_n$, while the endpoint terms are of order $\omega_n^{-4}$. A completely monotone kernel, on the other hand, is a positive combination of exponentials. For each exponential the correction in \cref{eq:finite-symbol-exact} has a fixed sign and is at most $\kappa_n/(2\pi n)$ times the symbol, and integration against $\nu\geq0$ preserves this bound. This positive-mixture argument is the mechanism behind the finite-horizon comparison, and hence behind the identity $\rho_{\mathfrak T}^{\mathrm{cos}}(k)=\rho(k)$. Complete monotonicity should therefore be viewed here as a sufficient structural condition for this finite-horizon rigidity, rather than as a generic consequence of positive type. Whether the same rigidity holds under weaker assumptions on the kernel remains open.
\end{remark}

\begin{remark}[Relation with the finite-horizon spectrum]\label{rem:kyfan}
Let
\begin{align*}
\mu_1(\Tend)\geq\mu_2(\Tend)\geq\cdots\geq0
\end{align*}
denote the eigenvalues of $\mathcal S_{k,\Tend}$ for a completely monotone kernel, repeated according to multiplicity and completed by zeros when necessary. Since $\{\psi_n\}_{n\geq1}$ is an orthonormal family, Ky Fan's variational principle gives, for every $N\geq1$,
\begin{align*}
\sum_{j=1}^{N}\mu_j(\Tend)
\geq
\sum_{n=1}^{N}\mathcal Q_{k,\Tend}[\psi_n].
\end{align*}
Consequently, \cref{thm:finite-symbol} yields
\begin{align*}
\sum_{j=1}^{N}\mu_j(\Tend)
\geq
\sum_{n=1}^{N}
\left(1-\frac{\kappa_n}{2\pi n}\right)
 m\!\left(\frac{2\pi n}{\Tend}\right).
\end{align*}
The cosine estimates therefore give lower bounds for partial sums of eigenvalues. They say nothing about individual eigenvalues; for their asymptotics we refer to Rosenblatt \cite{Rosenblatt1963} and Widom \cite{Widom1963,Widom1964}.
\end{remark}

\begin{remark}[Relation with classical Stieltjes-transform asymptotics]
\label{rem:stieltjes-index}
Let $\sigma$ be the image of the measure $\lambda\,\diff\nu(\lambda)$ under the map $\lambda\mapsto\lambda^{2}$. Then, for $x>0$,
\begin{align*}
 m(\sqrt{x})=\int_{[0,\infty)}\frac{1}{x+r}\,\diff\sigma(r).
\end{align*}
Thus, $x\mapsto m(\sqrt{x})$ is a Stieltjes transform. Classical Karamata theory characterises regularly varying Stieltjes transforms and confines their regular-variation index to $[-1,0]$; see, for example, \cite[Theorem~1.7.4]{BinghamGoldieTeugels1987}. In terms of the frequency variable $|\omega|=\sqrt{x}$, this corresponds to algebraic exponents in $[0,2]$. We do not use regular variation. $\rho(k)$ is defined for any kernel, whether or not $m$ is regularly varying.
\end{remark}

Note that $\rho$ measures the decay of $m$ itself, not of the gap weight $1-m(\omega)$, which tends to one. By itself, $\rho(k)$ carries only one-sided information; a Sobolev norm equivalence needs the two-sided comparison $m(\omega)\asymp |\omega|^{-\rho}$ (see \cref{rem:index-sobolev}). If
\begin{align*}
m(\omega)\asymp |\omega|^{-r}
\quad\text{as }|\omega|\to\infty,
\end{align*}
then $\rho(k)=r$. The pure constant kernel $\nu=c\delta_0$ remains a degenerate exception: its symbol vanishes for $\omega\neq0$ and the index is $+\infty$.

\begin{proposition}[First Bernstein moment and the endpoint $\rho=2$]\label{prop:first-moment-index}
Let $k$ be a non-constant completely monotone kernel and set
\begin{align*}
M_1:=\int_{[0,\infty)}\lambda\,\diff\nu(\lambda)\in(0,\infty].
\end{align*}
Then,
\begin{align}
\lim_{|\omega|\to\infty}\omega^2m(\omega)=M_1.
\label{eq:first-moment-limit}
\end{align}
If $M_1<\infty$, then
\begin{align*}
m(\omega)\sim M_1|\omega|^{-2},
\quad
\rho(k)=2,
\quad
M_1=-k'(0^+).
\end{align*}
If $M_1=\infty$, then $\omega^2m(\omega)\to\infty$; in this case the exponent $2$ is not attained in the defining set for $\rho(k)$, although its supremum may still equal $2$.
\end{proposition}

\begin{proof}
Identity \cref{eq:omega2m-rigidity} and the monotone convergence theorem give \cref{eq:first-moment-limit}. If $M_1<\infty$, this is equivalent to $m(\omega)\sim M_1|\omega|^{-2}$, and \cref{thm:finite-index-rigidity} gives $\rho(k)=2$. Differentiating Bernstein's representation for $t>0$ gives
\begin{align*}
-k'(t)=\int_{[0,\infty)}\lambda e^{-\lambda t}\diff\nu(\lambda),
\end{align*}
and monotone convergence as $t\downarrow0$ yields $-k'(0^+)=M_1$. The case $M_1=\infty$ follows directly from \cref{eq:first-moment-limit}.
\end{proof}

The decay of $m$ also determines which time regularity the memory dissipation sees: if $m(\omega)\asymp|\omega|^{-\rho}$, then the frequency weight in \cref{eq:freqrep} is comparable with the Sobolev weight $(1+\omega^{2})^{-\rho/2}$. For kernels of finite $L^{1}$-mass, the continuity and positivity of $m$ extend this comparison from high frequencies to the whole real line, as stated below.

\begin{remark}[Sobolev interpretation]\label{rem:index-sobolev}
If $\|k\|_{L^1(0,\infty)}<\infty$ and $m(\omega)\asymp|\omega|^{-\rho}$ with $\rho\in(0,2]$, continuity and strict positivity of $m$ imply
$m(\omega)\asymp(1+\omega^2)^{-\rho/2}$ on $\R$. Therefore, \cref{eq:freqrep} gives, for any $u\in\mathcal A_{\Tend}$, $\Dm[u](\Tend)\asymp\|W\|_{H^{-\rho/2}(\R;\Lp)}^2$. For the fractional kernel the exact homogeneous version is \cref{prop:fractional-frequency}, with weight $|\omega|^{-(1-\alpha)}$.
\end{remark}

\subsection{Weak-\texorpdfstring{$*$}{*} discontinuity at the instantaneous limit}
An $L^{1}$ kernel gives a compact operator on $L^{2}(0,\Tend)$, and an atom at the origin gives a multiple of the identity. Because the compact operators form a norm-closed set, a sequence of the former can converge to the latter strongly but never in norm.

\begin{theorem}[Compact-to-instantaneous obstruction]
\label{thm:compact-instantaneous}
Fix $\Tend>0$ and $a>0$. Let $(k_n)$ be a sequence of kernels such that each $k_n\in L^{1}(0,\Tend)$ is of positive type, and let $\mathcal S_n:=\mathcal S_{k_n,\Tend}$. If
\begin{align}
\mathcal S_n\phi\longrightarrow a\phi
\quad\text{in }L^{2}(0,\Tend)
\quad\text{for any }\phi\in L^{2}(0,\Tend),
\label{eq:strong-to-identity}
\end{align}
then
\begin{align}
\norm{\mathcal S_n-aI}_{\mathcal L(L^{2}(0,\Tend))}
\geq a
\quad\text{for any }n.
\label{eq:compact-distance}
\end{align}
In particular, $\mathcal S_n$ cannot converge to $aI$ in operator norm. Furthermore,
\begin{align}
\inf_{\substack{\phi\in L^2(0,\Tend)\\\phi\neq0}}
\frac{(\mathcal S_n\phi,\phi)}
{\norm{\phi}_{L^{2}(0,\Tend)}^{2}}
=0
\quad\text{for any }n,
\quad
\inf_{\substack{\phi\in L^2(0,\Tend)\\\phi\neq0}}
\frac{(aI\phi,\phi)}
{\norm{\phi}_{L^{2}(0,\Tend)}^{2}}
=a.
\label{eq:floor-jump-general}
\end{align}
Thus, strong convergence can preserve every fixed field while the uniform coercivity floor jumps from $0$ to $a$.
\end{theorem}

\begin{proof}
By \cref{prop:dichotomy}, each $\mathcal S_n$ is compact, self-adjoint and non-negative and has coercivity floor zero. Let $(\psi_j)_{j\geq1}$ be the orthonormal cosine sequence from \cref{eq:cosine-sequence}. For each fixed $n$, compactness gives $\mathcal S_n\psi_j\to0$ strongly as $j\to\infty$. Therefore,
\begin{align*}
\norm{(\mathcal S_n-aI)\psi_j}_{L^{2}(0,\Tend)}
\longrightarrow a,
\end{align*}
which proves \cref{eq:compact-distance}. The floor identity for $aI$ is immediate, and the remaining assertions follow.
\end{proof}

The estimate is the usual essential-norm bound for compact operators. The exponential kernels below show that equality in \cref{eq:compact-distance} occurs within the completely monotone class.

For comparison with the instantaneous limit, we extend the definition of the memory-symbol decay index to a non-negative real symbol $M$ by
\begin{align*}
\displaystyle
\rho(M)
:=
\sup\left\{
s\geq0:
\limsup_{|\omega|\to\infty}
|\omega|^sM(\omega)<\infty
\right\}.
\end{align*}
In particular, the constant symbol $M\equiv1$ has index $0$.

\begin{theorem}[Concentration of exponential memory at the origin]
\label{thm:wstar}
For $n\in\mathbb N$, let
\begin{align*}
k_n(t):=n e^{-nt},
\quad
\diff\mu_n(t):=k_n(t)\diff t,
\quad
\nu_n=n\delta_n.
\end{align*}
Then, the following assertions hold.
\begin{enumerate}[leftmargin=2.2em,label=\textnormal{(\alph*)}]
\item
The time measures satisfy
\begin{align}
\mu_n\rightharpoonup^*\delta_0
\quad\text{in }\mathcal M([0,\infty)).
\label{eq:mu-to-delta}
\end{align}
Furthermore, $(k_n)$ is a one-sided approximate identity for any Banach space $X$, any $1\leq p<\infty$, and any $F\in L^{p}(\mathbb R;X)$,
\begin{align}
k_n*F\longrightarrow F
\quad\text{in }L^{p}(\mathbb R;X),
\label{eq:causal-approx-id}
\end{align}
where $k_n$ is extended by zero to $(-\infty,0)$.

\item
For any fixed $\Tend>0$, let $\mathcal S_{n,\Tend}:=\mathcal S_{k_n,\Tend}$. Then,
\begin{align}
\mathcal S_{n,\Tend}\phi\longrightarrow\phi
\quad\text{in }L^{2}(0,\Tend)
\quad\text{for every }\phi\in L^{2}(0,\Tend),
\label{eq:finite-strong}
\end{align}
but
\begin{align}
\norm{\mathcal S_{n,\Tend}-I}_{\mathcal L(L^{2}(0,\Tend))}
=1
\quad\text{for every }n.
\label{eq:finite-not-norm}
\end{align}
Consequently, the coercivity floors are
\begin{align}
\inf_{\substack{\phi\in L^2(0,\Tend)\\\phi\neq0}}
\frac{(\mathcal S_{n,\Tend}\phi,\phi)}
{\norm{\phi}_{L^{2}(0,\Tend)}^{2}}
=0
\quad\text{for any }n,
\quad
\inf_{\substack{\phi\in L^2(0,\Tend)\\\phi\neq0}}
\frac{(I\phi,\phi)_{L^{2}(0,\Tend)}}
{\norm{\phi}_{L^{2}(0,\Tend)}^{2}}
=1.
\label{eq:exp-floor-jump}
\end{align}

\item
The memory symbol is
\begin{align}
m_n(\omega)=\frac{n^{2}}{n^{2}+\omega^{2}}.
\label{eq:mn}
\end{align}
For any fixed $n$, $m_n(\omega)\to0$ as $|\omega|\to\infty$ and $\rho(k_n)=2$, whereas for any fixed $\omega$, $m_n(\omega)\to1$. The limiting instantaneous symbol is $M_{\delta_0}\equiv1$ and $\rho(M_{\delta_0})=0$.

\item
The symbol convergence is not uniform:
\begin{align}
\norm{m_n-1}_{L^{\infty}(\mathbb R)}=1
\quad\text{for every }n.
\label{eq:wstar-not-uniform}
\end{align}
The high-frequency and instantaneous limits do not commute,
\begin{align}
\lim_{|\omega|\to\infty}\lim_{n\to\infty}m_n(\omega)
=
1
\neq
0
=
\lim_{n\to\infty}\lim_{|\omega|\to\infty}m_n(\omega),
\label{eq:limits-not-commute}
\end{align}
and the index jumps from $2$ to $0$.
\end{enumerate}
\end{theorem}

\begin{proof}
For \textnormal{(a)}, the substitution $r=nt$ gives
$\int_0^\infty\varphi(t)ne^{-nt}\diff t=\int_0^\infty\varphi(r/n)e^{-r}\diff r\to\varphi(0)$ for $\varphi\in\mathcal C_0([0,\infty))$. The $L^p$ assertion is the standard Banach-valued approximate-identity theorem, since $k_n\geq0$, $\int k_n=1$, and $\int_\delta^\infty k_n=e^{-n\delta}\to0$.

For \textnormal{(b)}, after zero extension to $\R$, $\mathcal S_{n,\Tend}\phi$ is convolution with $h_n(r)=\frac n2e^{-n|r|}$ restricted to $(0,\Tend)$, so the even approximate-identity theorem gives strong convergence. \Cref{thm:compact-instantaneous} yields $\|\mathcal S_{n,\Tend}-I\|\geq1$, whereas $0\leq\mathcal S_{n,\Tend}\leq I$ follows from $\|\mathcal S_{n,\Tend}\|=\Lambda_{k_n}(\Tend)\leq1-e^{-n\Tend}<1$; hence the norm is exactly one. The floor statement follows from \cref{prop:dichotomy}.

For \textnormal{(c)}, $\nu_n=n\delta_n$ gives $m_n(\omega)=n^2/(n^2+\omega^2)$ and $\rho(k_n)=2$ by \cref{prop:first-moment-index}. Pointwise $m_n\to1$, the instantaneous symbol has index zero, and \cref{eq:wstar-not-uniform,eq:limits-not-commute} follow directly from the explicit formula. This also proves \textnormal{(d)}.
\end{proof}

\begin{remark}[Interpretation of the instantaneous limit]\label{rem:wstar-interp}
For any $1\leq p<\infty$, part~\textnormal{(a)} gives $k_n*a_1(u(\cdot),v)\to a_1(u(\cdot),v)$ in $L^p(0,\Tend)$ for any strongly measurable $u:(0,\Tend)\to V$ and every fixed $v\in V$ such that $a_1(u(\cdot),v)\in L^p(0,\Tend)$, so the formal limit is $\partial_tu+(\mathsf A_0+\mathsf A_1)u=f$. At the symbol level,
$m_n(\omega)=1/(1+(\omega/n)^2)$, so each fixed frequency sees instantaneous action in the limit, while the transition sits at $|\omega|\asymp n$, where $m_n(n\xi)=1/(1+\xi^2)$ does not depend on $n$. Convergence of solutions and the jump of the coercivity constant in \cref{thm:wstar} are therefore compatible; compare \cite{ContiPataSquassina2006,Shikhman2026}.
\end{remark}

\section{Concluding remarks}
The main estimate of this paper is
\begin{align*}
\left(1-\frac{\kappa_n}{2\pi n}\right)m(2\pi n/\Tend)
\leq\mathcal Q_{k,\Tend}[\psi_n]\leq m(2\pi n/\Tend)
\end{align*}
for any locally integrable completely monotone kernel, any $\Tend>0$ and any $n$. It implies $\rho_{\Tend}^{\mathrm{cos}}(k)=\rho(k)\in[0,2]$ for non-constant kernels, and the Gaussian example shows that without complete monotonicity the boundary effects on $(0,\Tend)$ can change the exponent.

The other results are simpler but fit into the same picture: memory alone has coercivity constant zero on any bounded interval, the norm $\Lambda_k(\Tend)$ is continuous and strictly increasing with a unique crossing of one when the total mass exceeds one, and concentration at $t=0$ gives strong but not norm convergence. We have not considered nonlinear problems. It would be interesting to know whether $\rho$ plays a role for blow-up thresholds or lifespan estimates; the loss of uniform coercivity by itself does not imply blow-up.

% ------------------------------------------------------------
% Acknowledgements and declarations
% ------------------------------------------------------------
\paragraph{Funding.}
The author declares that no funds, grants, or other support were received during
the preparation of this manuscript.

\paragraph{Data availability.}
No datasets were generated or analysed during the current study.

\paragraph{Competing interests.}
The author declares no competing interests.

\paragraph{Use of AI tools.}
The author used an AI assistant (Claude, Anthropic) for language editing and for checking parts of the manuscript; all mathematical content and conclusions were verified by the author, who takes full responsibility for them.

% ------------------------------------------------------------
% Inline Springer-style references: no .bib file is required.
% Keep the entries alphabetically ordered by the first author's surname.
% ------------------------------------------------------------

\end{document}